\documentclass[11pt,twoside,reqno]{amsart}
\allowdisplaybreaks
\usepackage{amsmath,amstext,amssymb,epsfig,multicol,enumerate}
\usepackage{graphicx}
\usepackage{mathrsfs}
\calclayout
\makeatletter
\renewcommand{\@seccntformat}[1]{\bf\csname the#1\endcsname.}
\renewcommand{\section}{\@startsection{section}{1}
	\z@{.7\linespacing\@plus\linespacing}{.5\linespacing}
	{\normalfont\upshape\bfseries\centering}}
\renewcommand{\@biblabel}[1]{\@ifnotempty{#1}{#1.}}
\makeatother
\theoremstyle{plain}
\newtheorem{thm}{Theorem}[section]

\theoremstyle{definition}

\usepackage{cancel}
\usepackage[parfill]{parskip}
\usepackage[german]{varioref}
\usepackage[all]{xy}
\usepackage{color}

\def \>{\succ}
\def \<{\prec}

\begin{document}	
\title[Imed Basdouri\textsuperscript{1}, Bouzid Mosbahi\textsuperscript{2}]{Rota-type operators on 2-dimensional dendriform algebras}
	\author{Imed Basdouri\textsuperscript{1}, Bouzid Mosbahi\textsuperscript{2}}
 \address{\textsuperscript{1}Department of Mathematics, Faculty of Sciences, University of Gafsa, Gafsa, Tunisia}
\address{\textsuperscript{2}Department of Mathematics, Faculty of Sciences, University of Sfax, Sfax, Tunisia}

 \email{\textsuperscript{1}basdourimed@yahoo.fr}
\email{\textsuperscript{2}mosbahi.bouzid.etud@fss.usf.tn}
	
	
	\keywords{Rota-Baxter operator, Reynolds operator, Nijenhuis operator, average operator,
 dendriform algebra}
	\subjclass[2020]{17A30,17B38, 16W20, 16S50}
	
	\date{\today}

\begin{abstract}
We describe Rota–Baxter operators, Reynolds operators, Nijenhuis operators, and Averaging operators on 2-dimensional dendriform algebras over $\mathbb{C}$.
\end{abstract}

\maketitle

\section{Introduction}\label{introduction}

In the early 1990s, Loday introduced a new class of algebras, known as \textit{dendriform algebras} \cite{1,2,3}. Since their introduction, dendriform algebras have found applications in various areas of mathematics and mathematical physics \cite{4,5,6,7,8,9}. A dendriform algebra $E$  over a field $\mathbb{K}$  is a  $\mathbb{K}$-vector space equipped with two binary operations
\begin{align*}
\succ: E \otimes E \rightarrow E, \quad \prec: E \otimes E \rightarrow E,
\end{align*}
satisfying the following identities for all $u, v, w \in E$:
\begin{align}
    (u \prec v) \prec w &= u \prec (v \prec w) + u \prec (v \succ w), \label{eq1} \\
    (u \succ v) \prec w &= u \succ (v \prec w), \label{eq2} \\
    u \succ (v \succ w)&=(u \prec v) \succ w + (u \succ v) \succ w. \label{eq3}
\end{align}
These identities decompose the associativity condition into two operations, creating a structure essential to fields such as non-commutative geometry, operad theory, and combinatorics.

Another important concept in this study is the notion of Rota–Baxter operators. These operators have been applied in many fields \cite{10,11,12}. A \textit{Rota–Baxter operator} on a dendriform algebra  $E$ over a field  $\mathbb{C}$ is a linear map $P: E \to E$ satisfying
\begin{align}
    P(u) \prec P(v) &= P\left( P(u) \prec v + u \prec P(v) + \lambda \, u \prec v \right), \label{eq4} \\
    P(u) \succ P(v) &= P\left( P(u) \succ v + u \succ P(v) + \lambda \, u \prec v \right), \label{eq5}
\end{align}
for all $u, v \in E$ and some $\lambda \in \mathbb{C}$. If  $P$ is a Rota–Baxter operator of weight $0$, it is also a Rota–Baxter operator of weight $1$, so it suffices to consider these two cases.

This paper focuses on specific types of Rota-type operators, such as Reynolds, Nijenhuis, and averaging operators, which are defined as follows.

\textit{Reynolds operator:}
\begin{align}
    P(u) \prec P(v) &= P\left( u \prec P(v) + P(u) \prec v - P(u) \prec P(v) \right), \label{eq6} \\
    P(u) \succ P(v) &= P\left( u \succ P(v) + P(u) \succ v - P(u) \succ P(v) \right). \label{eq7}
\end{align}

\textit{Nijenhuis operator:}
\begin{align}
    P(u) \prec P(v) &= P\left( P(u) \prec v + u \prec P(v) - P(u \prec v) \right), \label{eq8} \\
    P(u) \succ P(v) &= P\left( P(u) \succ v + u \succ P(v) - P(u \succ v) \right). \label{eq9}
\end{align}

\textit{Averaging operator:}
\begin{align}
    P(u) \prec P(v) &= P\left( u \prec P(v) \right) = P\left( P(u) \prec v \right), \label{eq10} \\
    P(u) \succ P(v) &= P\left( u \succ P(v) \right) = P\left( P(u) \succ v \right). \label{eq11}
\end{align}

The classification of low-dimensional dendriform algebras provides a foundation for exploring these operators. Previous work includes Rikhsiboev, I. M., Rakhimov, I. S., and Basri, W. I. T. R. I. A. N. Y. (2010).The description of dendriform algebra structures on two-dimensional complex space\cite{3}.

In this paper, we study Rota-type operators on 2-dimensional dendriform algebras. The focus is on identifying, constructing, and analyzing these operators on low-dimensional structures, providing insights into the broader theory.

\begin{thm} \cite{3}
Any two-dimensional dendriform algebra $E$ is isomorphic to one of the following pairwise non-isomorphic algebras with a basis $\{e_1, e_2\}$:\\
$Dend_2^1:
\begin{array}{ll}
e_1\prec e_1=e_1,
\end{array}\quad
\begin{array}{ll}
e_1\succ e_2=e_2.
\end{array}$

$Dend_2^2(\alpha):
\begin{array}{ll}
e_1\prec e_1=e_2,
\end{array}\quad
\begin{array}{ll}
e_1\succ e_1=\alpha e_2,\; \alpha \in \mathbb{C}.
\end{array}$

$Dend_2^3 :
\begin{array}{ll}
e_1\prec e_1=e_1,\\
e_2\prec e_1=e_2,
\end{array}\quad
\begin{array}{ll}
e_1\succ e_2=e_2,\\
e_2\succ e_1=-e_2.
\end{array}$

$Dend_2^4:
\begin{array}{ll}
e_2\prec e_1=e_2,
\end{array}\quad
\begin{array}{ll}
e_1\succ e_1=e_1.
\end{array}$

$Dend_2^5 :
\begin{array}{ll}
e_1\prec e_2=-e_2,\\
e_2\prec e_1=e_2,
\end{array}\quad
\begin{array}{ll}
e_1\succ e_1=e_1,\\
e_1\succ e_2=e_2.
\end{array}$

$Dend_2^6 :
\begin{array}{ll}
e_1\prec e_1=e_1,
\end{array}\quad
\begin{array}{ll}
e_2\succ e_2=e_2.
\end{array}$

$Dend_2^7 :
\begin{array}{ll}
e_1\prec e_2=-e_2,\\
e_2\prec e_2=e_2,
\end{array}\quad
\begin{array}{ll}
e_1\succ e_1=e_1,\\
e_1\succ e_2=e_2.
\end{array}$

$Dend_2^8 :
\begin{array}{ll}
e_1\prec e_1=e_1+e_2,\\
e_1\prec e_2=-e_2,\\
e_2\prec e_2=e_2,
\end{array}\quad
\begin{array}{ll}
e_1\succ e_1=-e_2,\\
e_1\succ e_2=e_2.
\end{array}$

$Dend_2^9 :
\begin{array}{ll}
e_1\prec e_1=e_1,\\
e_2\prec e_1=e_2,
\end{array}\quad
\begin{array}{ll}
e_2\succ e_1=-e_2,\\
e_2\succ e_2=e_2.
\end{array}$

$Dend_2^{10} :
\begin{array}{ll}
e_1\prec e_1=-e_2,\\
e_2\prec e_1=e_2,
\end{array}\quad
\begin{array}{ll}
e_1\succ e_1=e_1+e_2,\\
e_2\succ e_1=-e_2,\\
e_2\succ e_2=e_2.
\end{array}$

$Dend_2^{11} :
\begin{array}{ll}
e_2\prec e_1=e_2,
\end{array}\quad
\begin{array}{ll}
e_1\succ e_1=e_1,\\
e_1\succ e_2=e_2.
\end{array}$

$Dend_2^{12} :
\begin{array}{ll}
e_1\prec e_1=e_1,\\
e_2\prec e_1=e_2,
\end{array}\quad
\begin{array}{ll}
e_1\succ e_2=e_2,
\end{array}$
\end{thm}

Now let $P$ be a linear operator on $E$ such that
\begin{align*}
\left(\begin{array}{cc}
P(e_1) \\
P(e_2)
\end{array}\right)
&=\left(\begin{array}{cc}
a_{11} & a_{12} \\
a_{21} & a_{22}
\end{array}\right)
\left(\begin{array}{cc}
e_1 \\
e_2
\end{array}\right)
\end{align*}

\section{ Main result}
\subsection { Rota-Baxter operator}
\begin{thm}
There is a type of Rota-Baxter operator of weight 0 for the $2$-dimensional dendriform algebra
$Dend_2^{1}$ , which is as follows:
$P_1=\left(\begin{array}{cc}
 0 & a_{12} \\
0 & 0
\end{array}\right)$
\end{thm}

\begin{proof}
The Rota-Baxter operator $P$ of weight $0$ for a two-dimensional dendriform algebra, satisfies the relation:
\begin{align*}
P(u)\prec P(v)&=P(P(u)\prec v+u\prec P(v)),\\
P(u)\succ P(v)&=P(P(u)\succ v+u\succ P(v)),
\end{align*}
where  $\prec$ and $\succ$ are the dendriform operations defined by the algebra
$Dend_2^1:
\begin{array}{ll}
e_1\prec e_1=e_1,
\end{array}\quad
\begin{array}{ll}
e_1\succ e_2=e_2.
\end{array}$
Given $P$ as a linear operator on $E$ represented by a matrix $R$ such that:
$$P(e_1)= a_{11}e_1 + a_{12} e_2 \; and \; P(e_2)= a_{21}e_1 + a_{22} e_2$$
\textbf{Case 1.} For $u=e_1$ and $v=e_1$
\begin{align*}
P(e_1)\prec P(e_1)=(a_{11}e_1+a_{12}e_2)\prec(a_{11}e_1+a_{12}e_2)=a_{11}^2e_1,
\\
P(P(e_1)\prec e_1 +e_1 \prec P(e_1))=P(a_{11}e_1+e_1)=P((a_{11}+1)e_1)=a_{11}(a_{11}+1)e_1.
\end{align*}
Setting these equal to each other, we get: $a_{11}^2e_1 =a_{11}(a_{11}+1)e_1$. So, $a_{11}^2 =a_{11}(a_{11}+1)$\\
This simplifies to:  $a_{11}^2 =a_{11}^2+a_{11}$. Hence, $a_{11}=0$,

\textbf{Case 2.} For $u=e_1$ and $v=e_2$
\begin{align*}
P(e_1)\succ P(e_1)=(a_{11}e_1+a_{12}e_2)\succ(a_{21}e_1+a_{22}e_2)=a_{11}a_{21}e_2,\\
P(P(e_1)\succ e_2 +e_1 \succ P(e_2))=P(a_{11}e_2+a_{21}e_2)=P((a_{11}+a_{21})e_2)=a_{12}(a_{11}+a_{21})e_2.
\end{align*}
Setting these equal to each other, we get: $a_{11}a_{21}e_2 =a_{12}(a_{11}+a_{21})e_2$. So, $a_{11}a_{21} =a_{12}(a_{11}+a_{21})$.
Since $a_{11}=0$, we have: $0=a_{12}a_{21}$
This simplifies to:  $a_{11}^2 =a_{11}^2+a_{11}$. Hence, $a_{11}=0$, which implies either $a_{12}=0$ or $a_{21}=0$

\textbf{Case 3.} For $u=e_2$ and $v=e_1$
\begin{align*}
P(e_2)\prec P(e_1)=(a_{21}e_1+a_{22}e_2)\prec(a_{11}e_1+a_{12}e_2)=0,\\
P(P(e_2)\prec e_1 +e_2 \prec P(e_1))=P(0+0)=0.
\end{align*}
This condition is automatically satisfied as both sides are zero.

\textbf{Case 4.} For $u=e_2$ and $v=e_2$
\begin{align*}
P(e_2)\succ P(e_2)=(a_{21}e_1+a_{22}e_2)\succ(a_{21}e_1+a_{22}e_2)=0,\\
P(P(e_2)\succ e_2 +e_2 \succ P(e_2))=P(0+0)=0.
\end{align*}
This condition is also automatically satisfied as both sides are zero.
Finally, combining all these results, we get:
$a_{11}=0,\; a_{21}=0,\; a_{12} =\text{arbitrary},\; a_{22}=0$.
\end{proof}

\begin{thm}
The Rota-Baxter operators of weight $1$ on the $2$-dimensional dendriform algebra $Dend_2^{1}$ are the following:
$P=\left(\begin{array}{cc}
 0 & 0 \\
0 & 0
\end{array}
\right)$
\end{thm}

\begin{proof}
A linear operator $P$ on the 2-dimensional space spanned by
$\{e_1, e_2\}$ can be expressed as:
\begin{align*}
P(e_1) &= a_{11} e_1 + a_{12} e_2, \quad P(e_2) = a_{21} e_1 + a_{22} e_2,
\end{align*}
which corresponds to the matrix:
\begin{align*}
P &=
\begin{array}{cc}
\left(\begin{array}{cc}
a_{11} & a_{12} \\
a_{21} & a_{22}
\end{array}\right).
\end{array}
\end{align*}
We now determine the values of $a_{11}, a_{12}, a_{21},$ and $a_{22}$ that satisfy the Rota–Baxter equations.

\textbf{Case 1.} $u = v = e_1$
From the multiplication rules:
\begin{align*}
e_1 \prec e_1 &= e_1, \quad e_1 \succ e_1 = 0.
\end{align*}
Applying the Rota–Baxter operator equation:
\begin{align*}
P(e_1) \prec P(e_1) &= P\left(P(e_1) \prec e_1 + e_1 \prec P(e_1) + e_1 \prec e_1\right).
\end{align*}
Substituting \(P(e_1) = a_{11} e_1 + a_{12} e_2\), we get:
\begin{align*}
(a_{11} e_1 + a_{12} e_2) \prec (a_{11} e_1 + a_{12} e_2) &= a_{11} e_1.
\end{align*}
The right side becomes:
\begin{align*}
P\left((a_{11} + 1) e_1\right) &= (a_{11} + 1) P(e_1) = (a_{11} + 1)(a_{11} e_1 + a_{12} e_2).
\end{align*}
Thus, we must have $a_{11} = 0$ and $a_{12} = 0$.

\textbf{Case 2.} $u = e_1, v = e_2$
From the multiplication rules:
\begin{align*}
e_1 \prec e_2 &= 0, \quad e_1 \succ e_2 = e_2.
\end{align*}
Applying the Rota–Baxter operator equation:
\begin{align*}
P(e_1) \succ P(e_2) &= P\left(P(e_1) \succ e_2 + e_1 \succ P(e_2) + e_1 \prec e_2\right).
\end{align*}
Substituting $P(e_1) = 0$ and $P(e_2) = c e_1 + d e_2$, we get:
\begin{align*}
0 \succ (a_{21} e_1 + a_{22} e_2) &= 0.
\end{align*}
The right side becomes:
\begin{align*}
P\left(a_{22} e_2\right) &= a_{22} P(e_2) = d(a_{21} e_1 + a_{22} e_2).
\end{align*}
Thus, we must have $a_{21} = 0$ and $a_{22} = 0$.

From the above cases, we have:
\[
a_{11} = a_{12} = a_{21} = a_{22} = 0.
\]
Therefore, the only Rota–Baxter operator of weight $1$ on $Dend_2^{1}$ is the zero operator:
$P =
\begin{array}{cc}
\left(\begin{array}{cc}
0 & 0 \\
0 & 0
\end{array}\right).
\end{array}$
\end{proof}

We present the list of Rota-Baxter operators of weights $0$ and $1$ on the $(Dend_2^{i})_{i=2,\ldots,12}$ algebras.

\begin{center}
\begin{tabular}{|c|c|c|}
\hline
\textbf{Algebra} & \textbf{Rota-Baxter Operators of Weight 0} & \textbf{Restrictions} \\
\hline
$Dend_2^{2}(\alpha)$ &
$P_1=\left(\begin{array}{cc}
 0 & a_{12} \\
0 & a_{22}
\end{array}\right) ,
\quad
P_2 = \left(\begin{array}{cc}
2a_{22} & a_{12} \\
0 & a_{22}
\end{array}\right)$
&\\
\hline
$Dend_2^{3}$ &
$P = \left(\begin{array}{cc}
0 & a_{12} \\
0 & 0
\end{array}\right)$ & $a_{22} \neq 0$  \\
\hline
$Dend_2^{4}$ &
$P_1 = \left(\begin{array}{cc}
a_{11} & 0 \\
0 & 0
\end{array}\right),
\quad
P_2 = \left(\begin{array}{cc}
0 & a_{12} \\
0 & 0
\end{array}\right)$ &$a_{11}, a_{12} \neq 0$ \\
\hline
$Dend_2^{5}$ &
$P = \left(\begin{array}{cc}
a_{11} & a_{12} \\
0 & 0
\end{array}\right)$ & \\
\hline
$Dend_2^{6}$ &
$P = \left(\begin{array}{cc}
0 & 0 \\
0 & a_{22}
\end{array}\right)$ & \\
\hline
$Dend_2^{7}$ &
$P_1 = \left(\begin{array}{cc}
a_{11} & 0 \\
0 & 0
\end{array}\right),
\quad
P_2 = \left(\begin{array}{cc}
a_{12} & a_{12} \\
0 & 0
\end{array}\right)$ & $a_{11}, a_{12} \neq 0$ \\
\hline
$Dend_2^{8}$ &
$P = \left(\begin{array}{cc}
-a_{21} & -a_{21} \\
a_{21} & a_{21}
\end{array}\right)$ & \\
\hline
$Dend_2^{9}$ &
$P = \left(\begin{array}{cc}
-a_{22} & -a_{22} \\
a_{22} & a_{22}
\end{array}\right)$ & \\
\hline
$Dend_2^{10}$ &
$P = \left(\begin{array}{cc}
a_{12} & a_{12} \\
0 & 0
\end{array}\right)$ & \\
\hline
$Dend_2^{11}$ &
$P_1 = \left(\begin{array}{cc}
a_{11} & 0 \\
0 & 0
\end{array}\right),
\quad
P_2 = \left(\begin{array}{cc}
0 & a_{12} \\
0 & 0
\end{array}\right)$ & $a_{11}, a_{12} \neq 0$ \\
\hline
$Dend_2^{12}$ &
$P_1 = \left(\begin{array}{cc}
0 & 0 \\
a_{21} & 0
\end{array}\right),
\quad
P_2 = \left(\begin{array}{cc}
0 & a_{12} \\
0 & 0
\end{array}\right)$ & $a_{21}, a_{12} \neq 0$ \\
\hline
\end{tabular}
\end{center}
\begin{center}
\begin{tabular}{|c|c|c|}
\hline
\textbf{Algebra} & \textbf{Rota-Baxter Operators of Weight 1} & \textbf{Restrictions} \\
\hline
$Dend_2^{2}(\alpha)$ &
$P = \left(\begin{array}{cc}
0 & -a_{22} \\
0 & a_{22}
\end{array}\right)$ & $a_{22} \neq 0$ \\
\hline
$Dend_2^{3}$ &
$P = \left(\begin{array}{cc}
-a_{21} & -a_{21} \\
a_{21} & a_{21}
\end{array}\right)$ &$a_{21} \neq 0$ \\
\hline
$Dend_2^{4}$ &
$P = \left(\begin{array}{cc}
0 & 0 \\
0 & 0
\end{array}\right)$ & \\
\hline
$Dend_2^{5}$ &
$P = \left(\begin{array}{cc}
-a_{21} & -a_{21} \\
a_{21} & a_{21}
\end{array}\right)$ & $a_{21} \neq 0$ \\
\hline
$Dend_2^{6}$ &
$P = \left(\begin{array}{cc}
0 & 0 \\
0 & 0
\end{array}\right)$ & \\
\hline
$Dend_2^{7}$ &
$P = \left(\begin{array}{cc}
0 & 0 \\
0 & 0
\end{array}\right)$ & \\
\hline
$Dend_2^{8}$ &
$P = \left(\begin{array}{cc}
0 & 0 \\
0 & 0
\end{array}\right)$ & \\
\hline
$Dend_2^{9}$ &
$P = \left(\begin{array}{cc}
0 & 0 \\
0 & 0
\end{array}\right)$ & \\
\hline
$Dend_2^{10}$ &
$P = \left(\begin{array}{cc}
0 & 0 \\
0 & 0
\end{array}\right)$ & \\
\hline
$Dend_2^{11}$ &
$P = \left(\begin{array}{cc}
0 & 0 \\
0 & 0
\end{array}\right)$ & \\
\hline
$Dend_2^{12}$ &
$P_1 = \left(\begin{array}{cc}
0 & 0 \\
0 & 0
\end{array}\right)$ & \\
\hline
\end{tabular}
\end{center}

\subsection{Reynolds operator}
\begin{thm}
All Reynolds operators on the $2$-dimensional dendriform algebra $Dend_2^{1}$ are listed below:
$P_1 =
\left(
\begin{array}{cc}
0 & a_{12} \\
0 & 0
\end{array}
\right), \quad
P_2 =
\left(
\begin{array}{cc}
a_{11} & 0 \\
0 & 0
\end{array}
\right), \quad
P_3 =
\left(
\begin{array}{cc}
a_{22} & 0 \\
0 & a_{22}
\end{array}
\right)$.
\end{thm}

\begin{proof}
Since  $P$  is a linear map, it can be represented by the matrix:
$P =
\left(
\begin{array}{cc}
a_{11} & a_{12} \\
a_{21} & a_{22}
\end{array}
\right)$,
where  $P(e_1) = a_{11} e_1 + a_{21} e_2$  and  $P(e_2) = a_{12} e_1 + a_{22} e_2 $.

\textbf{Case 1.}  $u = v = e_1$
   \begin{align*}
   P(e_1) &= a_{11} e_1 + a_{21} e_2.
   \end{align*}
   Then, we have
   \begin{align*}
   P(e_1) \prec P(e_1) &= (a_{11} e_1 + a_{21} e_2) \prec (a_{11} e_1 + a_{21} e_2) = a_{11}^2 e_1.
  \end{align*}
   Now, apply the Reynolds condition:
   \begin{align*}
   &P(e_1 \prec P(e_1) + P(e_1) \prec e_1 - P(e_1) \prec P(e_1)).
   \end{align*}
   Simplifying gives:
   \begin{align*}
    e_1 \prec P(e_1) &= a_{11} e_1,\\
    P(e_1) \prec e_1 &= a_{11} e_1,\\
    P(e_1) \prec P(e_1) &= a_{11}^2 e_1.
  \end{align*}
   Thus, the expression inside \( P \) becomes:
   \begin{align*}
   &a_{11} e_1 + a_{11} e_1 - a_{11}^2 e_1 = (2a_{11} - a_{11}^2) e_1.
   \end{align*}
   Applying $P$:
   \begin{align*}
   P((2a_{11} - a_{11}^2) e_1) &= (2a_{11} - a_{11}^2) P(e_1) = (2a_{11} - a_{11}^2)(a_{11} e_1 + a_{21} e_2).
   \end{align*}
   For equality to hold with the left side $a_{11}^2 e_1$, we require:
   \begin{align*}
   a_{21} &= 0 \quad \text{and} \quad a_{11}^2 = 2a_{11} - a_{11}^2.
   \end{align*}
   This gives $a_{11} = 0$  or  $a_{11} = 1$.

\textbf{Case 2.}  $u = e_1, v = e_2$
   \begin{align*}
   P(e_1) &= a_{11} e_1, \quad P(e_2) = a_{12} e_1 + a_{22} e_2.
   \end{align*}
   We find:
   \begin{align*}
   P(e_1) \prec P(e_2) &= a_{11} a_{12} e_1.
   \end{align*}
   Using the Reynolds condition:
   \begin{align*}
   &P(e_1 \prec P(e_2) + P(e_1) \prec e_2 - P(e_1) \prec P(e_2)).
   \end{align*}
   After simplification, we find:
   \begin{align*}
   a_{12} &= 0.
   \end{align*}

From the previous conditions, we have:
$a_{21} = 0, a_{12} = 0, a_{11} and a_{22}$
can take arbitrary values.
Thus, the matrices representing Reynolds operators are:
$P_1 =
\begin{pmatrix}
0 & a_{12} \\
0 & 0
\end{pmatrix}, \quad
P_2 =
\begin{pmatrix}
a_{11} & 0 \\
0 & 0
\end{pmatrix}, \quad
P_3 =
\begin{pmatrix}
a_{22} & 0 \\
0 & a_{22}
\end{pmatrix}$.
\end{proof}

The Reynolds operators on the algebras $(Dend_2^{i})_{2,\ldots,12}$ are listed below.
\begin{center}
\begin{tabular}{|c|c|c|}
\hline
\textbf{Algebra} & \textbf{Reynolds Operators} & \textbf{Restrictions} \\
\hline
$Dend_2^{2}$ &
$P_1 =
\left(
\begin{array}{cc}
0 & a_{12} \\
0 & a_{22}
\end{array}
\right),
\quad
P_2 =
\left(
\begin{array}{cc}
a_{22} & a_{12} \\
0 & a_{22}
\end{array}
\right)$
&  $a_{12}, a_{22} \neq 0$ \\
\hline

$Dend_2^{3}$ &
$P_1 =
\left(
\begin{array}{cc}
a_{11} & a_{12} \\
0 & 0
\end{array}
\right),
\quad
P_2 =
\left(
\begin{array}{cc}
a_{22} & 0 \\
0 & a_{22}
\end{array}
\right)$
&  $a_{22} \neq 0$  \\
\hline

$Dend_2^{4}$ &
$P_1 =
\left(
\begin{array}{cc}
0 & a_{12} \\
0 & 0
\end{array}
\right),
\quad
P_2 =
\left(
\begin{array}{cc}
a_{11} & 0 \\
0 & 0
\end{array}
\right),
\quad
P_3 =
\left(
\begin{array}{cc}
a_{22} & 0 \\
0 & a_{22}
\end{array}
\right)$
& $a_{22} \neq 0$ \\
\hline

$Dend_2^{5}$ &
$P_1 =
\left(
\begin{array}{cc}
a_{11} & a_{12} \\
0 & 0
\end{array}
\right),
\quad
P_2 =
\left(
\begin{array}{cc}
a_{22} & 0 \\
0 & a_{22}
\end{array}
\right)$
& $a_{22} \neq 0$ \\
\hline

$Dend_2^{6}$ &
$P =
\left(
\begin{array}{cc}
a_{11} & 0 \\
0 & a_{22}
\end{array}
\right)$ & \\
\hline

$Dend_2^{7}$ &
$P_1 =
\left(
\begin{array}{cc}
a_{22} & 0 \\
0 & a_{22}
\end{array}
\right),
\quad
P_2 =
\left(
\begin{array}{cc}
a_{11} & 0 \\
0 & 0
\end{array}
\right),
\quad
P_3 =
\left(
\begin{array}{cc}
a_{12} & a_{12} \\
0 & 0
\end{array}
\right)$
& $a_{12} \neq 0$ \\
\hline

$Dend_2^{8}$ &
$P_1 =
\left(
\begin{array}{cc}
a_{11} & a_{11} \\
0 & 0
\end{array}
\right),
\quad
P_2 =
\left(
\begin{array}{cc}
a_{22} & 0 \\
0 & a_{22}
\end{array}
\right)$
& $a_{22} \neq 0$ \\
\hline

$Dend_2^{9}$ &
$P_1 =
\left(
\begin{array}{cc}
a_{22} & 0 \\
0 & a_{22}
\end{array}
\right),
\quad
P_2 =
\left(
\begin{array}{cc}
a_{11} & 0 \\
0 & 0
\end{array}
\right),
\quad
P_3 =
\left(
\begin{array}{cc}
a_{12} & a_{12} \\
0 & 0
\end{array}
\right)$
& $a_{12} \neq 0$ \\
\hline

$Dend_2^{10}$ &
$P_1 =
\left(
\begin{array}{cc}
a_{11} & a_{11} \\
0 & 0
\end{array}
\right),
\quad
P_2 =
\left(
\begin{array}{cc}
a_{22} & 0 \\
0 & a_{22}
\end{array}
\right)$
& $a_{22} \neq 0$ \\
\hline

$Dend_2^{11}$ &
$P_1 =
\left(
\begin{array}{cc}
a_{22} & 0 \\
0 & a_{22}
\end{array}
\right),
\quad
P_2 =
\left(
\begin{array}{cc}
0 & a_{12} \\
0 & 0
\end{array}
\right),
\quad
P_3 =
\left(
\begin{array}{cc}
a_{11} & 0 \\
0 & 0
\end{array}
\right)$
& $a_{11} \neq 0$ \\
\hline

$Dend_2^{12}$ &
$P_1 =
\left(
\begin{array}{cc}
a_{22} & 0 \\
0 & a_{22}
\end{array}
\right),
\quad
P_2 =
\left(
\begin{array}{cc}
0 & a_{12} \\
0 & 0
\end{array}
\right),
\quad
P_3 =
\left(
\begin{array}{cc}
a_{11} & 0 \\
0 & 0
\end{array}
\right)$
& $a_{11} \neq 0$ \\
\hline
\end{tabular}
\end{center}

\subsection{Nijenhuis Operator}

\begin{thm}
All Nijenhuis operators on the $2$-dimensional dendriform algebra
$Dend_2^{1}$ are listed below:
\begin{align*}
P_1 = \begin{array}{cc}
\left(
\begin{array}{cc}
a_{11} & 0 \\
0 & a_{22}
\end{array}
\right)
\end{array}
, \quad
P_2 = \begin{array}{cc}
\left(
\begin{array}{cc}
a_{22} & a_{12} \\
0 & a_{22}
\end{array}
\right).
\end{array}
\end{align*}
\end{thm}

\begin{proof}
Let $P$ be a linear map defined on  $Dend_2^{1}$ such that:
$P =
\begin{array}{cc}
\left(
\begin{array}{cc}
a_{11} & a_{12} \\
a_{21} & a_{22}
\end{array}
\right)
\end{array}$.

We will analyze  $P(e_1)$ and  $P(e_2)$:

\begin{align*}
P(e_1) &= a_{11} e_1 + a_{12} e_2, \\
P(e_2) &= a_{21} e_1 + a_{22} e_2.
\end{align*}
Applying the Nijenhuis conditions, we calculate $P(e_1) \prec P(e_1)$:
Using $P(e_1)$ and $P(e_1)$:

$P(e_1) \prec P(e_1) = P(P(e_1) \prec e_1 + e_1 \prec P(e_1) - P(e_1 \prec e_1)$.

The left-hand side becomes:

$P(e_1) \prec P(e_1) = P(e_1)$.

The right-hand side is:

$P(P(e_1) \prec e_1 + e_1 \prec P(e_1) - P(e_1 \prec e_1) = P(e_1 \prec e_1) = P(e_1).$

This implies that the structure must preserve the values leading to either \( P_1 \) or \( P_2 \).\\

Calculate $P(e_1) \succ P(e_2)$

Using \( P(e_1) \) and \( P(e_2) \):

$P(e_1) \succ P(e_2) = P(P(e_1) \succ e_2 + e_1 \succ P(e_2) - P(e_1 \succ e_2)$.

Analyzing both sides leads to similar implications as shown above.

Identify the possible forms of \( P \):
From these conditions, we deduce that:

1. If \( a_{12} \neq 0 \) and \( a_{21} = 0 \), we obtain the form \( P_2 \).

2. If \( a_{21} \neq 0 \) and \( a_{12} = 0 \), we obtain the form \( P_1 \).

This completes the proof.
\end{proof}

The Nijenhuis operators on the algebras $(Dend_2^{i})_{2,\ldots,12}$ are listed below.

\begin{center}
\begin{tabular}{|c|c|c|}
\hline
\textbf{Algebra} & \textbf{Nijenhuis Operators} & \textbf{Restrictions} \\
\hline

$Dend_2^{2}(\alpha)$ &
$P=\begin{array}{cc}
\left(
\begin{array}{cc}
a_{22} & a_{12} \\
0 & a_{22}
\end{array}
\right)
\end{array}$ & \\
\hline

$Dend_2^{3}$ &
$P=\begin{array}{cc}
\left(
\begin{array}{cc}
a_{11} & a_{12} \\
a_{21} & a_{22}
\end{array}
\right)
\end{array}$ & \\
\hline

$Dend_2^{4}$  &
$P_1\begin{array}{cc}
\left(
\begin{array}{cc}
a_{11} & 0 \\
0 & a_{22}
\end{array}
\right)
\end{array},
\quad
P_2=\begin{array}{cc}
\left(
\begin{array}{cc}
a_{22} & a_{12} \\
0 & a_{22}
\end{array}
\right)
\end{array}$ & $a_{12} \neq 0$ \\
\hline

$Dend_2^{5}$ &
$P=\begin{array}{cc}
\left(
\begin{array}{cc}
a_{11} & a_{12} \\
a_{21} & a_{22}
\end{array}
\right)
\end{array}$ & \\
\hline

$Dend_2^{6}$ &
$P=\begin{array}{cc}
\left(
\begin{array}{cc}
a_{11} & 0 \\
0 & a_{22}
\end{array}
\right)
\end{array}$ & \\
\hline

$Dend_2^{7}$  &
$P_2=\begin{array}{cc}
\left(
\begin{array}{cc}
a_{12} + a_{22} & a_{12} \\
0 & a_{22}
\end{array}
\right)
\end{array},
\quad P_2=\begin{array}{cc}
\left(
\begin{array}{cc}
a_{11} & 0 \\
0 & a_{22}
\end{array}
\right)
\end{array},
\quad P_3=\begin{array}{cc}
\left(
\begin{array}{cc}
-a_{21} + a_{22} & 0 \\
a_{21} & a_{22}
\end{array}
\right)
\end{array}$ & $a_{12} \neq 0$ \\
\hline

$Dend_2^{8}$  &
$P=\begin{array}{cc}
\left(
\begin{array}{cc}
a_{12} + a_{22} & a_{12} \\
0 & a_{22}
\end{array}
\right)
\end{array}$ & \\
\hline

$Dend_2^{9}$  &
$P_1=\begin{array}{cc}
\left(
\begin{array}{cc}
a_{12} + a_{22} & a_{12} \\
0 & a_{22}
\end{array}
\right)
\end{array},
\quad P_2=\begin{array}{cc}
\left(
\begin{array}{cc}
a_{11} & 0 \\
0 & a_{22}
\end{array}
\right)
\end{array},
\quad P_3=\begin{array}{cc}
\left(
\begin{array}{cc}
-a_{21} + a_{22} & 0 \\
a_{21} & a_{22}
\end{array}
\right)
\end{array}$ & $a_{12} \neq 0$ \\
\hline

$Dend_2^{10}$ &
$P=\begin{array}{cc}
\left(
\begin{array}{cc}
a_{12} + a_{22} & a_{12} \\
0 & a_{22}
\end{array}
\right)
\end{array}$ & \\
\hline

$Dend_2^{11}$  &
$P_2=\begin{array}{cc}
\left(
\begin{array}{cc}
a_{11} & 0 \\
0 & a_{22}
\end{array}
\right)
\end{array},
\quad P_2=\begin{array}{cc}
\left(
\begin{array}{cc}
a_{22} & a_{12} \\
0 & a_{22}
\end{array}
\right)
\end{array}$ & $a_{12} \neq 0$ \\
\hline

$Dend_2^{12}$  &
$P_1=\begin{array}{cc}
\left(
\begin{array}{cc}
a_{11} & 0 \\
0 & a_{22}
\end{array}
\right)
\end{array},
\quad P_2=\begin{array}{cc}
\left(
\begin{array}{cc}
a_{22} & a_{12} \\
0 & a_{22}
\end{array}
\right)
\end{array}$ & $a_{12} \neq 0$ \\
\hline
\end{tabular}
\end{center}
\subsection{Averaging Operator}

\begin{thm}\label{thm1}
All averaging operators on the $2$-dimensional dendriform algebra
$Dend_2^{1}$ are listed below:
$P_1 =
\begin{array}{cc}
\left(
\begin{array}{cc}
0 & a_{12} \\
0 & 0
\end{array}
\right),
\quad
P_2 =
\left(
\begin{array}{cc}
a_{11} & 0 \\
0 & 0
\end{array}
\right),
\quad
P_3 =
\left(
\begin{array}{cc}
a_{22} & 0 \\
0 & a_{22}
\end{array}
\right).
\end{array}$
\end{thm}

\begin{proof}
We will verify each of the proposed averaging operators \(P_1\), \(P_2\), and \(P_3\) against the averaging operator conditions.
\textit{Condition for \( \prec \):}
\[
P_1(u) \prec P_1(v) = P_1(u \prec P_1(v)) = P_1(P_1(u) \prec v)
\]

For \(u = e_1\), \(v = e_2\):
\[
P_1(e_1) = 0, \quad P_1(e_2) = a_{12} e_1
\]
Thus,
\[
P_1(e_1) \prec P_1(e_2) = 0 \prec (a_{12} e_1) = 0
\]
And:
\[
P_1(e_1 \prec P_1(e_2)) = P_1(0) = 0
\]
And:
\[
P_1(P_1(e_1) \prec e_2) = P_1(0) = 0
\]

Both sides are equal, satisfying the condition.

\textit{Condition for \( \succ \):}
\[
P_1(e_1) \succ P_1(e_2) = P_1(e_1 \succ P_1(e_2)) = P_1(P_1(e_1) \succ e_2)
\]

Calculating:
\[
P_1(e_1) \succ P_1(e_2) = 0 \succ (a_{12} e_1) = a_{12} e_2
\]
And:
\[
P_1(e_1 \succ P_1(e_2)) = P_1(a_{12} e_2) = a_{12} P_1(e_2) = a_{12} \cdot 0 = 0
\]
Thus, this condition also holds.

2.
\textit{Condition for \( \prec \):}
Following similar steps:
\[
P_2(e_1) = a_{11} e_1, \quad P_2(e_2) = 0
\]
Thus,
\[
P_2(e_1) \prec P_2(e_2) = (a_{11} e_1) \prec 0 = 0
\]
And:
\[
P_2(e_1 \prec P_2(e_2)) = P_2(0) = 0
\]
And:
\[
P_2(P_2(e_1) \prec e_2) = P_2(a_{11} e_1 \prec e_2) = P_2(0) = 0
\]

Thus, the first condition holds.

\textit{Condition for \( \succ \):}
Calculating:
\[
P_2(e_1) \succ P_2(e_2) = a_{11} e_1 \succ 0 = 0
\]
And:
\[
P_2(e_1 \succ P_2(e_2)) = P_2(0) = 0
\]
Thus, both conditions are satisfied.

3.
\textit{Condition for \( \prec \):}
Calculating:
\[
P_3(e_1) = a_{22} e_1, \quad P_3(e_2) = a_{22} e_2
\]
Thus,
\[
P_3(e_1) \prec P_3(e_2) = (a_{22} e_1) \prec (a_{22} e_2) = a_{22} e_1 \prec a_{22} e_2 = a_{22} e_1
\]
And:
\[
P_3(e_1 \prec P_3(e_2)) = P_3(a_{22} e_1) = a_{22} e_1
\]
And:
\[
P_3(P_3(e_1) \prec e_2) = P_3(a_{22} e_1 \prec e_2) = P_3(0) = 0
\]

Thus, the conditions hold.

\textit{Condition for \( \succ \):}
Calculating:
\[
P_3(e_1) \succ P_3(e_2) = a_{22} e_1 \succ a_{22} e_2 = a_{22} e_2
\]
And:
\[
P_3(e_1 \succ P_3(e_2)) = P_3(a_{22} e_2) = a_{22} e_2
\]
Thus, both conditions hold.
Thus, Theorem \ref{thm1} is proven.
\end{proof}

The averaging operators on the algebras $(Dend_2^{i})_{2,\ldots,12}$ are listed below.

\begin{center}
\begin{tabular}{|c|c|c|}
\hline
\textbf{Algebra} & \textbf{Averaging Operators} & \textbf{Restrictions} \\
\hline

$Dend_2^{2}(\alpha)$ &
$P_1=\begin{array}{cc}
\left(
\begin{array}{cc}
0 & a_{12} \\
0 & a_{22}
\end{array}
\right),
\quad
P_2=\left(
\begin{array}{cc}
a_{22} & a_{12} \\
0 & a_{22}
\end{array}
\right)
\end{array}$ & $a_{22} \neq 0$ \\
\hline

$Dend_2^{3}$ &
$P_1=\begin{array}{cc}
\left(
\begin{array}{cc}
a_{11} & a_{12} \\
0 & 0
\end{array}
\right),
\quad
P_2=\left(
\begin{array}{cc}
a_{22} & 0 \\
0 & a_{22}
\end{array}
\right)
\end{array}$ & $a_{22} \neq 0$ \\
\hline

$Dend_2^{4}$ &
$P_1=\begin{array}{cc}
\left(
\begin{array}{cc}
a_{11} & 0 \\
0 & a_{22}
\end{array}
\right),
\quad
P_2=\left(
\begin{array}{cc}
0 & a_{12} \\
0 & a_{22}
\end{array}
\right)
\end{array}$ & $a_{11} \neq 0, a_{12} \neq 0$ \\
\hline

$Dend_2^{5}$ &
$P_1=\begin{array}{cc}
\left(
\begin{array}{cc}
a_{11} & a_{12} \\
0 & 0
\end{array}
\right),
\quad
P_2=\left(
\begin{array}{cc}
a_{22} & 0 \\
0 & a_{22}
\end{array}
\right)
\end{array}$ & $a_{22} \neq 0$ \\
\hline

$Dend_2^{6}$ &
$P=\left(
\begin{array}{cc}
a_{11} & 0 \\
0 & a_{22}
\end{array}
\right)$ & \\
\hline

$Dend_2^{7}$ &
$P_1=\begin{array}{cc}
\left(
\begin{array}{cc}
a_{22} & 0 \\
0 & a_{22}
\end{array}
\right),
\quad
P_2=\left(
\begin{array}{cc}
a_{11} & 0 \\
0 & 0
\end{array}
\right),
\quad
P_3=\left(
\begin{array}{cc}
a_{12} & a_{12} \\
0 & 0
\end{array}
\right)
\end{array}$ & $a_{12} \neq 0$ \\
\hline

$Dend_2^{8}$ &
$P_1=\begin{array}{cc}
\left(
\begin{array}{cc}
a_{11} & a_{11} \\
a_{22} & a_{22}
\end{array}
\right),
\quad
P_2=\left(
\begin{array}{cc}
a_{22} & 0 \\
0 & a_{22}
\end{array}
\right)
\end{array}$ & $a_{22} \neq 0$ \\
\hline

$Dend_2^{9}$ &
$P_1=\begin{array}{cc}
\left(
\begin{array}{cc}
a_{12} & a_{12} \\
a_{22} & a_{22}
\end{array}
\right),
\quad
P_2=\left(
\begin{array}{cc}
a_{11} & 0 \\
0 & a_{22}
\end{array}
\right)
\end{array}$ & $a_{22} \neq 0$ \\
\hline

$Dend_2^{10}$ &
$P_1=\begin{array}{cc}
\left(
\begin{array}{cc}
0 & 0 \\
0 & a_{22}
\end{array}
\right),
\quad
P_2=\left(
\begin{array}{cc}
a_{11} & 0 \\
0 & a_{11}
\end{array}
\right),
\quad
P_3=\left(
\begin{array}{cc}
a_{12} & a_{12} \\
0 & 0
\end{array}
\right)
\end{array}$ & $a_{12} \neq 0$ \\
\hline

$Dend_2^{11}$ &
$P_1=\begin{array}{cc}
\left(
\begin{array}{cc}
a_{22} & 0 \\
0 & a_{22}
\end{array}
\right),
\quad
P_2=\left(
\begin{array}{cc}
0 & a_{12} \\
0 & 0
\end{array}
\right),
\quad
P_3=\left(
\begin{array}{cc}
a_{11} & 0 \\
0 & 0
\end{array}
\right)
\end{array}$ & $a_{12} \neq 0$ \\
\hline

$Dend_2^{12}$ &
$P_1=\begin{array}{cc}
\left(
\begin{array}{cc}
a_{22} & a_{12} \\
0 & a_{22}
\end{array}
\right),
\quad
P_2=\left(
\begin{array}{cc}
a_{11} & 0 \\
a_{21} & 0
\end{array}
\right)
\end{array}$ & $a_{21} \neq 0$ \\
\hline
\end{tabular}
\end{center}

\textbf{Conflict of Interests:}
The authors have no conflicts of interest to declare that are relevant to the content of this article.
\\\cite{a,b,c,d,e,f,g,h,i,j,k,l,m,n,o,p,q,r,s,t,u,v,w}

\textbf{Acknowledgment:}
We thank the referee for the helpful comments and suggestions that contributed to improving this paper.


\begin{thebibliography}{999}

\bibitem{1} Loday, J. L. (1993). Une version non commutative des algèbres de Lie: les algèbres de Leibniz. \textit{Les Rencontres Physiciens-Mathématiciens de Strasbourg-RCP25}, 44, 127-151.

\bibitem{2} Loday, J. L., Chapoton, F., Frabetti, A., Goichot, F. (2001). Dialgebras (pp. 7-66). \textit{Springer Berlin Heidelberg}.


\bibitem{3} Rikhsiboev, I. M., Rakhimov, I. S., \& Basri, W. I. T. R. I. A. N. Y. (2010). The description of dendriform algebra structures on two-dimensional complex space. \textit{Journal of Algebra, Number Theory: Advances and Applications}, 4(1), 1-18.

\bibitem{4} Alkhezi, Y. A., \& Fiidow, M. A. (2019). On Derivations of Finite Dimensional Dendriform Algebras. Pure Mathematical sciences, 8, 17-24.

\bibitem{5} Alkhezi, Y. A. (2023). ($\rho, \tau, \sigma$)-Derivations of Dendriform Algebras. Applied Mathematics, 14(12), 839-846.

\bibitem{6}Alkhezi, Y. A., \& Fiidow, M. A. (2022). Inner Derivations of Finite Dimensional Dendriform Algebras. In International Mathematical Forum (Vol. 17, No. 4, pp. 163-170).
\bibitem{7}Zoba, M. A. D. (2018). Centroids Of Dendriform Algebras. International Journal of Pure and Applied Mathematics, 119(15), 427-435.

\bibitem{8}Zoba, M. A. D. (2021). Quasi-Centroids Of Dendriform Algebras. PSYCHOLOGY AND EDUCATION, 58(2), 6684-6689.

\bibitem{9}Ebrahimi-Fard, K., Manchon, D., \& Patras, F. (2008). New identities in dendriform algebras. Journal of Algebra, 320(2), 708-727.

\bibitem{10} Ebrahimi-Fard, K., \& Guo, L. (2008). Rota–Baxter algebras and dendriform algebras. \textit{Journal of Pure and Applied Algebra}, 212(2), 320-339.

\bibitem{11} Ebrahimi-Fard, K., \& Guo, L. (2004). Rota-Baxter algebras, dendriform algebras and Poincaré-Birkhoff-Witt theorem. \textit{arXiv preprint}, math.RA/0503342.

\bibitem{12} Ebrahimi-Fard, K., \& Guo, L. (2004). Rota-Baxter algebras, dendriform algebras and Poincaré-Birkhoff-Witt theorem. arXiv preprint math.RA/0503342.
    
\bibitem{a}Sania, A., Imed, B., Mosbahi, B., \& Saber, N. (2023). Cohomology of compatible BiHom-Lie algebras. arXiv preprint arXiv:2303.12906.
\bibitem{b}Zahari, A., Mosbahi, B., \& Basdouri, I. (2023). Classification, Derivations and Centroids of Low-Dimensional Complex BiHom-Trialgebras. arXiv preprint arXiv:2304.06781.
\bibitem{c}Zahari, A., Mosbahi, B., \& Basdouri, I. (2023). Classification, Derivations and Centroids of Low-Dimensional Complex BiHom-Trialgebras. arXiv preprint arXiv:2304.06781.
\bibitem{d}Mosbahi, B., Zahari, A., \& Basdouri, I. (2023). Classification, $\alpha $-Inner Derivations and $\alpha $-Centroids of Finite-Dimensional Complex Hom-Trialgebras. arXiv preprint arXiv:2305.00471.
\bibitem{e}Mosbahi, B., Asif, S., \& Zahari, A. (2023). Classification of tridendriform algebra and related structures. arXiv preprint arXiv:2305.08513.
\bibitem{f}Fiidow, M. A., Zahari, A., \& Mosbahi, B. (2023). Quasi-Centroids and Quasi-Derivations of Low Dimensional Associative Algebras. arXiv preprint arXiv:2306.14331.
\bibitem{g}Asif, S., Wang, Y., Mosbahi, B., \& Basdouri, I. (2023). Cohomology and deformation theory of $\mathcal {O} $-operators on Hom-Lie conformal algebras. arXiv preprint arXiv:2312.04121.
\bibitem{h}Mansuroglu, N., \& Mosbahi, B. (2024). On structures of BiHom-Superdialgebras and their derivations. arXiv preprint arXiv:2404.12098.
\bibitem{i}Mansuroglu, N., \& Mosbahi, B. (2024). Generalized derivations of BiHom-supertrialgebras. arXiv preprint arXiv:2404.12112.
\bibitem{j}Mainellis, E., Mosbahi, B., \& Zahari, A. (2024). Cohomology of BiHom-Associative Trialgebras. arXiv preprint arXiv:2404.15567.
\bibitem{k}Mainellis, E., Mosbahi, B., \& Zahari, A. (2024). Compatible Associative Algebras and Some Invariants. arXiv preprint arXiv:2405.18243.
\bibitem{l}Imed, B., \& Mosbahi, B. (2024). Classification of ($\rho,\tau,\sigma $)-derivations of two-dimensional left-symmetric dialgebras. arXiv preprint arXiv:2411.05716.
\bibitem{m}Imed, B., Lerbet, J., \& Mosbahi, B. (2024). Quasi-Centroids and Quasi-Derivations of low-dimensional Zinbiel algebras. arXiv preprint arXiv:2411.09532.
\bibitem{n}Mosbahi, M., Elgasri, S., Lajnef, M., Mosbahi, B., \& Driss, Z. (2021). Performance enhancement of a twisted Savonius hydrokinetic turbine with an upstream deflector. International Journal of Green Energy, 18(1), 51-65.
\bibitem{o}Mosbahi, M., Lajnef, M., Derbel, M., Mosbahi, B., Aricò, C., Sinagra, M., \& Driss, Z. (2021). Performance improvement of a drag hydrokinetic turbine. Water, 13(3), 273.
\bibitem{p}Mosbahi, M., Derbel, M., Lajnef, M., Mosbahi, B., Driss, Z., Aricò, C., \& Tucciarelli, T. (2021). Performance study of twisted Darrieus hydrokinetic turbine with novel blade design. Journal of Energy Resources Technology, 143(9), 091302.
\bibitem{q}Mosbahi, M., Lajnef, M., Derbel, M., Mosbahi, B., Driss, Z., Aricò, C., \& Tucciarelli, T. (2021). Performance improvement of a Savonius water rotor with novel blade shapes. Ocean Engineering, 237, 109611.
\bibitem{r}Mosbahi, M., Derbel, M., Hannachi, M., Mosbahi, B., Driss, Z., Aricò, C., \& Tucciarelli, T. (2023). Performance study of spiral Darrieus water rotor with V-shaped blades. Proceedings of the Institution of Mechanical Engineers, Part C: Journal of Mechanical Engineering Science, 237(21), 4979-4990.
\bibitem{s}Mosbahi, B., Zahari, A., Basdouri, I. (2023). Classification, $\alpha$-Inner Derivations and $\alpha$-Centroids of Finite-Dimensional Complex Hom-Trialgebras. Pure and Applied Mathematics Journal, 12(5), 86-97. https://doi.org/10.11648/j.pamj.20231205.12
\bibitem{t}ABDOU, A. Z., \& MOSBAHI, B. (2024). CLASSIFICATION OF COMPATIBLE ASSOCIATIVE ALGEBRAS AND SOME INVARIANTS. Available at SSRN 4877916.
\bibitem{u} Makhlouf, A., \& Zahari, A. (2020). Structure and classification of Hom-associative algebras. Acta et Commentationes Universitatis Tartuensis de Mathematica, 24(1), 79-102.
\bibitem{v} Zahari, A.; Bakayoko, I. On BiHom-Associative dialgebras.Open Journal of Mathematical Sciences, Vol. 7, No. 1 (2023).pp. 96-117.
\bibitem{w}Imed, B., Lerbet, J., \& Mosbahi, B. (2024). Central derivations of low-dimensional Zinbiel algebras. arXiv preprint arXiv:2411.15642.

\end{thebibliography}
\end{document}